\documentclass[12pt]{article}

\usepackage{latexsym}
\usepackage{a4wide}

\newtheorem{theorem}{Theorem}[section]
\newtheorem{lemma}[theorem]{Lemma}

\newenvironment{proof}{\rmfamily\upshape\mdseries\small{\noindent\normalsize\bfseries Proof:}}{\nopagebreak\rule{2mm}{2mm}\newline}

\newcommand{\Magma}{{\scshape Magma}}

\def\N{{\bf N}}
\def\C{{\bf C}}
\def\Q{{\bf Q}}
\def\R{{\bf R}}
\def\GF(#1){{\rm GF(}#1{\rm )}}
\def\ip<#1,#2>{{\langle#1,#2\rangle}}
\def\Irr{{\rm Irr}}
\def\IBr{{\rm IBr}}

\def\G{{\mathcal{G}}}
\def\E{{\mathcal{E}}}
\def\fru{{\sqrt{-1}}}
\def\m_i{{m_\R}}
\def\split{\!:\!}
\def\club{($\ast$)}
\def\qe{quasi-elementary}

\begin{document}
\title{An algorithm for computing Schur indices of characters}
\author{
W. R. Unger\thanks{Supported by ARC grants DP0452427, DP0772368} \\
School of Mathematics and Statistics, \\
University of Sydney, \\
Sydney, Australia \\
}
\maketitle

\begin{abstract}
We describe an algorithm for computing Schur
indices of irreducible characters of a finite group $G$,
based on computations within $G$ and its subgroups and with
their character tables.
The algorithm has been implemented within \Magma\ and
examples of its use are given. We also compare the output of the algorithm
on groups associated with sporadic simple groups with the results of W. Feit.
The differences are considered.
\end{abstract}

\section{Introduction}\label{sec:intro}
The notion of Schur index was introduced by Schur in 1906.
Suppose that $G$ is a finite group and $\chi$ an absolutely irreducible
character of $G$. Let $F$ be a field and let $\psi$ be the sum of the Galois
conjugates of $\chi$ over $F$, so $\psi$ takes values in $F$.
The minimal integer $m>0$ such that $m\psi$ is the
character of an $F$-representation of $G$ is the
\emph{Schur index of $\chi$ over $F$}, denoted $m_F(\chi)$.
When the characteristic of $F$ is not zero we have $m_F(\chi)=1$,
so we will only consider characteristic zero representations and their
characters when asking for a Schur index.

An alternative approach to the Schur index is by representation fields.
There exists a field $E$ extending $F(\chi)$ such that there is an
$E$-representation of $G$ affording $\chi$ and with $[E:F(\chi)] = m_F(\chi)$.
Furthermore, $m_F(\chi)$ is the minimum degree of any such field extension.

For further introduction to Schur indices see Chapter 10 of \cite{Isaacs}
or Chapter 38 of \cite{Huppert}. The above statements about the Schur index
are covered in both these references.

Since 1906 there has been much work on Schur indices published. A connection
with central simple algebras is laid out in Chapter 38 of \cite{Huppert},
and for number fields $F$ (also for $p$-adic number fields), the Schur index
$m_F(\chi)$ is the order of the corresponding element of the Brauer group of 
$F(\chi)$. In this context it is natural to consider Schur indices over
completions of $F(\chi)$, i.e. real numbers, complex numbers,
and $p$-adic number fields, in order to get the Schur index over the global
field. These local fields are easier to cope with.
A fundamental result of R.~Brauer \cite{Brauer51} and
E.~Witt \cite{Witt52} reduces us to considering groups of a special kind,
the \qe\ groups. This combination of reductions led to methods to compute
$m_F(\chi)$ in results from the 1960s to the 1980s.

With the availability of character tables for many simple groups in the ATLAS
of finite groups \cite{ATLAS} came the question of Schur indices of these
characters.
W.~Feit computed $p$-adic Schur indices for characters related to the
sporadic simple groups in \cite{Feit83,Feit96}.
Feit's main tool was the ordinary character table of each group.

Another use for a Schur indices is when constructing characteristic 0
matrix representations of a group. Combined with the character table, 
Schur indices can be used when proving
representations irreducible for instance, and give information on minimal
degree representation fields.

Despite this continuing interest in Schur indices, the author has been unable
to find any published algorithm for computing Schur indices of group characters.
This work describes such an algorithm and its implementation in
\Magma\ \cite{MagmaLanguage}.
The algorithm described here was implemented in \Magma\ V2.16, and returns
the Schur indices of the character $\chi$ over all completions of the rational
numbers $\Q$.
The initial approach is Feit's character table methods, then moving on to
consider \qe\ subgroups in the cases where this fails.
Feit's results for sporadic simple group characters have been checked using
the \Magma\ implementation, and some corrections are given here.

The author is part of the \Magma\ group's effort to check the character tables
of the ATLAS using \Magma's character table algorithm \cite{UngerChtr}.
In parallel with computing the character tables, Schur indices are also
being found and are available along with the character tables in \Magma\ V2.22.

This work began at the suggestion of John Cannon and Claus Fieker and I thank
them for their encouragment.
Claus Fieker has written a program which takes the character,
the Schur indices computed here, and a number field, and computes the Schur
index of the character over the number field.
I would also like to thank Gabi Nebe for much help with this project during two
visits to Sydney.

%%%%%%%%%%%%%%%%%%%%%%%%%%%%%%%%%%%%%%%%
\section{Notation and Preliminaries}\label{sec:prelim}
Let $G$ be a finite group and $\chi\in\Irr(G)$, where $\Irr(G)$ denotes the
set of ordinary irreducible characters of $G$. Let $N$ be the exponent of $G$
and let $F$ be any field of characteristic zero. 
We will use $\zeta_n$ to denote a primitive $n$th root of unity.
For any positive integer $n$, $F(n)$ will be written to denote $F(\zeta_n)$. 
Let $F(\chi)\subseteq F(N)$ denote the
extension of $F$ generated by all the $\chi(g)$, $g\in G$.
The notation $\ip<\alpha,\beta>$ will be used for the inner product of
characters. We will always use $q$ to denote a possible prime divisor
of some Schur index, and $p$ to denote some place of \Q.
For integers $a$ and $b$, $(a,b)$ will denote their greatest common divisor.
The degree of a field extension $E/F$ will be denoted $[E:F]$.
We will use $m(\chi)$ to denote the Schur index over 
\Q\ and $m_p(\chi)$ for the Schur index over the $p$-adic numbers $\Q_p$.

We collect a number of results on Schur indices.
The following are among those stated in Chapter 10 of \cite{Isaacs}
and Chapter 38 of \cite{Huppert}.

\begin{theorem}\label{globalm} Suppose that $G$ is a finite group,
$\chi\in\Irr(G)$, and that $F$ is a field of characteristic zero.
\begin{enumerate}
\item $m_F(\chi) = m_{F(\chi)}(\chi)$.
\item If $E$ is any extension of $F$, then $m_E(\chi)$ divides $m_F(\chi)$.
\item If $E$ is a finite extension of $F$, then $m_F(\chi)$ divides
$m_E(\chi)[E:F]$.
\item If $\psi$ is the character of any $F$-representation of $G$, then
$m_F(\chi)$ divides $\ip<\chi,\psi>$.
\end{enumerate}
\end{theorem}

The next result records upper bounds for $m(\chi)$, most found in the works 
cited above.
The names attached are attributions given by these sources.

\begin{theorem}\label{globalU}
Suppose that $G$ is a finite group with exponent $N$,
$\chi\in\Irr(G)$, and that $F$ is a field of characteristic zero. Let
$m = m(\chi)$.
\begin{enumerate}
\item $m$ divides $\chi(1)$.
\item (Feit) $m$ divides $|G|/\chi(1)$.
\item (Fein-Yamada) $m$ divides $N$.
%\item (Fein) For all positive integers $n$, $m$ divides $n\ip<\chi^n,1>$.
%\item (Brauer-Speiser) If $\chi = \bar\chi$ (i.e. $\chi$ is real),
%then $m$ divides $2$.
\item (Benard-Schacher) The field $\Q(\chi)$ contains a primitive $m$th root
of unity.
%\item (Brauer) If $F=F(N)$ then $m_F(\chi) = 1$.
\item (Solomon)
$m$ divides $\varphi(2p_1p_2\cdots p_\ell)$, where $p_1,p_2,\dots p_\ell$
are the distinct primes dividing the order of $G$.
%\item (Goldschmidt-Isaacs) If $q$ is a prime such that the
%Sylow $q$-subgroup, $S$, of the Galois group
%$\G(F(N)/F(\chi))$ is cyclic, and either $q$ is odd, or $S=1$, or
%$\fru\in F(\chi)$, then $q$ does not divide $m_F(\chi)$.
\item (Roquette) If $G$ is $q$-group and either $q$ is odd or $\fru\in F(\chi)$,
then $m_F(\chi) = 1$.
\item If $q$ is a prime such that the Sylow q-subgroups of $G$ are
elementary abelian, then $q$ does not divide $m$.
\item If $q^a$ is a prime power divisor of $m$ with $q^a > 2$,
then there exists a prime $p$ such that $q^a$ divides $p-1$ and $G$ contains
an element of order $pq^a$.
\end{enumerate}
\end{theorem}
\begin{proof}
The last item is the only one not given in the two texts cited.
For this see \cite{Schmid85}, section 7, application 3.
\end{proof}

The results above are aimed at showing that $m_F(\chi)$ is small,
and it seems that it is quite rare to find a character $\chi$ that
has $m(\chi)>2$.
On the other hand, it has been shown that every positive integer arises as
the Schur index of some character. 
For a relatively simple construction of a character having Schur index $n > 1$
over $\Q$, let $p$ be a prime such that $n$ divides $p-1$ but n is coprime
to $(p-1)/n$ (there are infinitely many such primes as any $p$ congruent
to $n+1$ modulo $n^2$ will do), take $r$ to have order $n$
modulo $p$, and let
$$G = \langle x,y \,\vert\, x^p=1, y^{n^2}=1, y^{-1}xy = x^r \rangle.$$
Take $C$ to be the subgroup of $G$ generated by $x$ and $y^n$. Then $C$
is a cyclic normal subgroup of $G$, and we take $\chi = \lambda^G$, where 
$\lambda$ is any faithful linear character of $C$. It follows from Lemma~3
of \cite{Turull} that $m(\chi) = n$. In \cite{Lorenz72} there is a construction
for a character with Schur index $n$ over a given number field.

%%%%%%%%%%%%%%%%%%%%%%%%%%%%%%%%%%%%%%
\section{Reductions}
\subsection{Subgroups and their characters}\label{sec:sub}
A standard approach to Schur indices is to examine characters of subgroups
of $G$. In the early 1950s Brauer and Witt both showed that Schur index
computations for a general finite group could be reduced to the case where
the group is of special type, a \emph{\qe} group.

If $q$ is a prime, we say that a finite group $G$ is \emph{\qe\ at $q$} when
$G$ has a cyclic normal subgroup $A$ such that
$G/A$ is a $q$-group. This is equivalent to $G$ having a cyclic normal
$q$-complement. We say that $G$ is \emph{\qe} when $G$
is \qe\ at $q$ for some prime $q$.
The following lemma gives some useful easy facts about \qe\ groups and
their characters.
\begin{lemma}\label{elementary}
Let $q$ be a rational prime and suppose that $G$ is \qe\ at $q$.
\begin{enumerate}
\item Every subgroup and quotient group of $G$ is \qe\ at $q$.
\item If $\chi\in\Irr(G)$ then both $\chi(1)$ and $m_F(\chi)$ are powers of $q$.
\end{enumerate}
\end{lemma}

The next lemma is (2A) of \cite{Brauer51}. It shows how
characters of subgroups can be useful when computing Schur indices.
\begin{lemma}\label{subgpq}
Let $G$ be a finite group with $\chi\in\Irr(G)$, let $F$ be a field
of characteristic zero, and fix a rational prime $q$. Let $U$ be a subgroup
of $G$ and $\eta\in\Irr(U)$ be a character such that $q$ does not divide 
$\ip<\eta^G, \chi>[F(\chi,\eta):F(\chi)]$. 
Then the $q$-part of $m_F(\chi)$ equals the $q$-part of $m_{F(\chi)}(\eta)$.
\end{lemma}
%\begin{proof}
%Without loss of generality we may assume that $F= F(\chi)$.
%By Lemma (10.4) of \cite{Isaacs} we have
%$m_F(\chi)$ divides $m_F(\eta)\ip<\eta^G, \chi>[F(\chi,\eta):F(\chi)]$, and
%so the $q$-part of $m_F(\chi)$ divides the $q$-part of $m_F(\eta)$.
%On the other hand, $m_F(\chi)\chi$ is the character of an $F$-representation
%of $G$, and we restrict this to $H$ and take the inner product with $\eta$
%to see that
%$m_F(\eta)$ divides $m_F(\chi)\ip<\chi_H,\eta> = m_F(\chi)\ip<\eta^G, \chi>$.
%Thus the $q$-part of $m_F(\eta)$ divides the $q$-part of $m_F(\chi)$.
%\end{proof}

We now state the fundamental theorem of Brauer and Witt.
\begin{theorem}[Brauer-Witt] \label{BW}
Let $G$ be a finite group with $\chi\in\Irr(G)$, let $F$ be a field
of characteristic zero, and fix a rational prime $q$.
There exists a \qe\ $U\le G$ and
$\eta\in\Irr(U)$ such that $q$ does not divide
$\ip<\eta^G, \chi>[F(\chi,\eta):F(\chi)]$.
If $q$ divides $m_F(\chi)$, then we may take $U$ to be \qe\ at $q$.
\end{theorem}
This theorem is due to Brauer \cite{Brauer51} and Witt \cite{Witt52}.
The last sentence follows from Lemma~\ref{elementary}.

In his article proving the above result, Brauer (\textit{op. cit.}) noted
that the pair $(U, \eta)$ found when $F = \Q$ is universal for all 
characteristic 0 fields $F$.

The difficulty with using this theory is locating a useful subgroup and
character. Feit \cite{Feit83} expressly avoids using \qe\ subgroups
in his computation of Schur indices except as a last resort.
We will use subgroups at some point in our algorithm, so we try to make 
the best of it. Brauer's proof showed that a suitable subgroup
$U$ could be found amongst the maximal \qe\ subgroups of $G$.
The following lemma extends that idea.
\begin{lemma} \label{supergroup}
Let $U\le G$ be a subgroup and $\eta\in\Irr(U)$ be a character such that
$q$ does not divide $\ip<\eta^G, \chi>[F(\chi,\eta):F(\chi)]$.
Suppose that $V\le G$ is a subgroup such that some conjugate of $U$ is
contained in $V$.
Then there exists $\xi\in\Irr(V)$ such that $q$ does not divide
$\ip<\xi^G, \chi>[F(\chi,\xi):F(\chi)]$.
\end{lemma}
\begin{proof}
We may assume that $F = F(\chi)$.
Take $g\in G$ such that $U^g\le V$.
Define $\eta^g\in\Irr(U^g)$ by  $\eta^g(u^g) = \eta(u)$ for all $u\in U$.
We have $\eta^G = (\eta^g)^G$ and $F(\eta) = F(\eta^g)$, so
we may replace $U$ by $U^g$ and $\eta$ by $\eta^g$ and assume that $U\le V$.

Let $A$ be the Galois group of $F(N)/F$, where $N$ is the exponent
of $G$.
The group $A$ is abelian and has a unique Sylow $q$-subgroup $S$.
Let $L$ be the fixed field of $S$.
Since $q$ does not divide $[F(\eta):F]$, we have $F(\eta) \subseteq L$.
Thus for all $\sigma\in S$ we have $\eta^\sigma = \eta$ and so,
for all $\psi\in\Irr(V)$, $\ip<\eta^V, \psi> = \ip<\eta^V, \psi^\sigma>$.
Let $\Omega$ be the set of all $S$-orbits on $\Irr(V)$ and for each
$\omega\in\Omega$ let $a_\omega = \ip<\eta^V, \psi>$, where $\psi\in\omega$.
We have
\begin{equation} \label{ip1}
\eta^V = \sum_{\omega\in\Omega} a_\omega(\sum\omega).
\end{equation}
We take the inner product of $\eta^G$ with $\chi$.
\begin{equation} \label{ip2}
\ip<\eta^G,\chi> = \ip<\eta^V, \chi_V> =
\sum_{\omega\in\Omega} a_\omega\ip<\sum\omega, \chi_V>.
\end{equation}
Thus we can choose $\omega\in\Omega$
such that $q$ does not divide $a_\omega\ip<\sum\omega, \chi_V>$
and we take $\xi$ to be an element of this $S$-orbit. 
Since $\chi$ is $S$-invariant we have
\begin{equation} \label{ip3}
\ip<\sum\omega, \chi_V> = |\omega|\ip<\xi, \chi_V>,
\end{equation}
so $q$ does not divide $|\omega|$. But $\omega$ is an orbit
of a $q$-group, hence $|\omega|=1$ and $\xi$ is $S$-invariant.
It follows that $F(\xi) \subseteq L$, so $q$ does not divide $[F(\xi):F]$. 
Finally, equation (\ref{ip3}) also shows that $q$ does not divide
$\ip<\xi, \chi_V> = \ip<\xi^G, \chi>$.
\end{proof}

\subsection{Faithful Characters}
A reduction we may make when computing Schur indices is to insist
that the character be faithful. This involves no change to the Schur index,
as the following lemma states.
\begin{lemma}\label{faithful}
Let $\hat G = G/\ker\chi$ and let $\hat\chi$ be the faithful
character of $\hat G$ that lifts to $\chi$. Then for all $F$
we have $m_F(\chi) = m_F(\hat\chi)$.
\end{lemma}
This follows as a representation of $G$ affording $\chi$ is a representation
of $\hat G$ affording $\hat\chi$.
We will reduce to quotients of subgroups to obtain a faithful character
(as $\hat\chi$) and compute Schur indices of this character.
A number of the parts of Theorem~\ref{globalU} may give sharper results when
applied to $\hat\chi$ rather than $\chi$.

\section{Schur Indices over complete fields}\label{sec:local}
In the chapter on Schur Indices of \cite{Huppert}, the connection between
the Schur index of a character and a certain central simple algebra is set out.
The theory of central simple algebras then gives much information
about the Schur index problem considered here.
There is considerable advantage to be gained by considering Schur indices over
$p$-adic fields. The following result is fundamental to our approach.
It follows from Corollary 18.6 of \cite{Pierce}.
\begin{theorem}\label{local-global}
If $F$ is an algebraic number field, then $m_F(\chi)$ is the least common
multiple of the Schur indices $m_{K}(\chi)$, where $K$ runs over 
the completions of $F(\chi)$ at all places.
\end{theorem}

We have effective methods to compute the Schur index in the case of
an Archimedean completion. All Schur indices over the complex field are 1.
The following was proved by Frobenius and Schur in 1906 and gives an
effective algorithm for computing Schur indices over the real numbers.
\begin{theorem}[Frobenius-Schur] \label{FS}
For $\chi\in\Irr(G)$, define the indicator $i(\chi)$ by
$$i(\chi) = \frac{1}{|G|}\sum_{x\in G} \chi(x^2).$$
There are exactly three possible values for $i=i(\chi)$:
\begin{enumerate}
\item $i=1$. In this case $\R(\chi) = \R$ and $\m_i(\chi) = 1$.
\item $i=0$. In this case $\R(\chi) = \C$ and $\m_i(\chi) = 1$.
\item $i=-1$. In this case $\R(\chi) = \R$ and $\m_i(\chi) = 2$.
\end{enumerate}
\end{theorem}
A proof is obtained by combining (4.5), (4.15), and (4.19) of \cite{Isaacs}.

To achieve our program, we will make Schur index calculations
over the $p$-adic fields $\Q_p$.
\begin{theorem} \label{ANF-local}
\begin{enumerate} 
\item (Benard) Let $F$ be an algebraic number field contained in a cyclotomic
field and let $p$ be any finite rational place.
If $P$ and $Q$ are places of $F$ over $p$ then the Schur indices of $\chi$
over the completions $F_P$ and $F_Q$ are equal.
\item If $K$ is a local field with $K = K(\chi)$,
$E/K$ is an extension of degree $k$,
and if $d$ is the greatest common divisor of $k$ and $m_K(\chi)$,
then $m_K(\chi) = dm_E(\chi)$.
\end{enumerate}
\end{theorem}
The first result is Theorem~1 of \cite{Benard72}. 
The second follows from Proposition~17.10 of \cite{Pierce}.

In view of this theorem we concentrate on computing $m_p(\chi)$, as these
values, together with the degrees of the extensions
\hbox{$[E:\Q_p(\chi)]$},
suffice to compute Schur indices over all algebraic number fields.

\begin{theorem} \label{localU1}
Let $G$ be a finite group with $\chi\in\Irr(G)$.
\begin{enumerate}
\item If $p$ is an odd prime, then $m_p(\chi)$ divides $p-1$.
\item $m_2(\chi)$ divides $2$, and $m_K(\chi) = 1$ when $K = \Q_2(4)$.
\item If $G$ has abelian Sylow 2-subgroups, then $m_2(\chi) = 1$.
\end{enumerate}
\end{theorem}
Parts 1 and 2 are results that have been derived a number of times,
perhaps first by Witt \cite{Witt52}, Satz 10 and Satz 11.
Yamada \cite{Yamada74} proves these results and much more.

The third part is found in \cite{Yamada78}. This paper also gives a
construction, giving, for any odd prime $p$ and any $d$ dividing $p-1$,
a character $\chi$ with $m_p(\chi) = d$.
The group of this construction has cyclic Sylow subgroups,
so the third part does not extend to odd primes%
\footnote{There is an error in the defining relations of the group of the
construction given in \cite{Yamada78}. The following group works:
$\langle x,y \,\vert\, x^p=1, y^{d(p-1)}=1, y^{-1}xy = x^r \rangle$,
where $r$ has order $p-1$ modulo $p$.}.

Proposition 18.5 of \cite{Pierce} implies that the set of places giving
Schur index $>1$ is finite, but it is helpful to have an explicit list
containing these places. For this, Brauer character theory is useful.
For fixed finite $p$ we will consider the $p$-modular Brauer characters
of $G$. Let $B$ be a $p$-block of $G$. The set of ordinary
irreducible characters in $B$ will be denoted $\Irr(B)$ and the set of
irreducible Brauer characters in $B$ will be denoted $\IBr(B)$.
The restriction of an ordinary character $\psi$ to the $p$-regular elements
of $G$ will be written $\psi^*$.
\begin{theorem} \label{localU2} We fix a finite rational prime $p$ and
let $B$ be the $p$-block with $\chi\in\Irr(B)$.
\begin{enumerate}
\item Let $\phi\in\IBr(G)$ occur with multiplicity $k$ as a
constituent of $\chi^*$. Then
$m_p(\chi)$ divides $k\,[\Q_p(\chi,\phi):\Q_p(\chi)]$.
\item If $p$ does not divide $|G|/\chi(1)$, then $m_p(\chi) = 1$
\item (Feit) Suppose that $x$ is a $p$-regular element of
$G$ and that $\psi(x)\in\Q_p(\chi)$ for all $\psi\in\Irr(B)$.
Then $m_p(\chi)$ divides $\chi(x)$ in the ring of algebraic integers.
\end{enumerate}
\end{theorem}
These results are stated in \cite{Feit83}.
The second part follows from the first as we may take $\phi = \chi^*$
by Theorem~3.18 of \cite{Navarro}.
Feit proves the third part from the first in \cite{Feit83},
Theorem~3.1 and Corollary 3.2.
Feit attributes the first part to R.~Brauer, and refers to his own text
\cite{Feit82}, Theorem~IV.9.3, for a proof.
On the other hand, Benard, in \cite{Benard76},
attributes it to ``K.~Kronstein in some unpublished work'',
and gives a proof (\textit{op. cit.} Theorem~2.6).

We see from the second part that the $p$-adic Schur indices we need to
consider are limited to those primes $p$ dividing the order of $G$.

Stronger results have been proved in the case where the block
containing $\chi$ has cyclic defect groups.
The following is proved in \cite{Benard76}, Theorems~8.1 and 8.15,
also in \cite{Feit82}, Chapter VII, \S13.
Remark (III.13) of \cite{Plesken} gives a generalisation.
\begin{theorem}[Benard's Formula] \label{BenardFeit}
Let $p$, $B$, and $\chi$ be as in Theorem~\ref{localU2} and further suppose
that $B$ has cyclic defect groups.
\begin{enumerate}
\item The fields $\Q_p(\phi)$, for $\phi\in\IBr(B)$, are equal.
Call this field $K$. 
\item $m_p(\chi) = [K(\chi):\Q_p(\chi)]$.
\end{enumerate}
\end{theorem}
Generally it can be difficult to find irreducible Brauer characters,
but to apply this theorem we have no need to.
The field $K$ in the above can be obtained from the ordinary character table.
We have
$$K = \Q_p(\phi:\phi\in\IBr(B)) = \Q_p(\psi^*:\psi\in\Irr(B)),$$
the first equality by the first part of the theorem, the second equality as
the two sets of characters are integer linear combinations of one-another.
We will often apply Theorem~\ref{BenardFeit} when $G$ is \qe, hence soluble,
in which case $\IBr(G)$ is easily obtained from $\Irr(G)$.

We give another result on local Schur indices that is useful
for computations. The following is stated in \cite{Feit83} (2.15),
and follows from, for instance, Proposition 18.7a of \cite{Pierce}
(Hasse's Sum Theorem).
\begin{theorem}\label{counting}
Let $V$ be the set of all places of $\Q(\chi)$.
For $v\in V$ let $m_v(\chi)$ denote the Schur index of $\chi$ over the
completion of $\Q(\chi)$ at $v$.
For prime $q$, suppose that $q^e>1$ is the $q$-part of $m(\chi)$, and let
$V_q =\{v\in V\;\vert\;q^e\;\mbox{divides}\; m_v(\chi)\}$.
Then $|V_q| > 1$, and, if $q^e=2$, then $|V_q|$ is even.
\end{theorem}
Note that when $v$ lies over the rational place $p$, $m_v(\chi) = m_p(\chi)$,
since $\Q_p(\chi)$ is a completion of $\Q(\chi)$, and
now apply Theorem~\ref{ANF-local}.

\subsection{Schur indices of $2$-group characters}
When $G$ is a $q$-group for $q$ an odd prime, the result of Roquette quoted as
part of Theorem~\ref{globalU} tells us that we have $m(\chi)=1$.

Now suppose that $G$ is a $2$-group. The quaternion group of order 8
has a character $\chi\in\Irr(G)$ with $m(\chi)=2$,
so Roquette's result does not generalise to $2$-groups.
We have seen that $m_p(\chi)=1$ for all odd primes $p$, $m_2(\chi)$ divides
$2$, and $\m_i(\chi)$ is determined by the Frobenius-Schur indicator.
We show how to determine $m_2(\chi)$.
\begin{lemma} \label{twogps}
Let $G$ be a finite $2$-group with $\chi\in\Irr(G)$.
If $\Q(\chi)=\Q$ then $m_2(\chi)=\m_i(\chi)$, otherwise $m_2(\chi) = 1$.
\end{lemma}
\begin{proof}
The method here is as \cite{Yamada79}.
Note that $\Q(\chi)\subseteq\Q(2^n)$. Since $2$ is totally ramified in
$\Q(2^n)$ (\cite{LangANT} p73 Theorem~1), $2$ is totally ramified in
$\Q(\chi)$, so $\Q(\chi)$ has a unique place over $2$. It follows that
$m_2(\chi)$ may be determined from $\m_i(\chi)$ and $\Q(\chi)$ by using
Theorem~\ref{counting}.
We observe that the degree of $\Q(\chi)$ over $\Q$ is a power of $2$.

First suppose that $\Q(\chi)=\Q$. In this case
$m_2(\chi) = \m_i(\chi)$ since $\Q(\chi)$ has one place at $2$
and one at infinity.

Now suppose that $\Q(\chi)\not=\Q$.
If the indicator of $\chi$ is $0$ or $1$ then $\m_i(\chi)=1$,
and we have $m_2(\chi)=1$.
When the indicator is $-1$, $\m_i(\chi)=2$, $\Q(\chi)$ is totally real, and
has $[\Q(\chi):\Q]$ places at infinity. This degree is a
non-trivial power of $2$, hence is even, so we again have $m_2(\chi) = 1$.
\end{proof}

%%%%%%%%%%%%%%%%%%%%%%%%%%%%%%%%%%%%%%%%
\section{Computing 2-adic Schur Indices}\label{sec:compute_m_p}
In this section we discuss a method to deal with the last remaining
problem when computing the $q$-part of $m_p(\chi)$ for $p,q$ rational
primes and $\chi\in\Irr(G)$, where $G$ is a \qe\ group.
We first summarise methods gathered to date.

If $G$ is \qe\ at $r$ then the Schur index of $\chi$ is a power of $r$, so
if $G$ is not \qe\ at $q$ then this $q$-part is $1$,
and we will now assume that $G$ is \qe\ at $q$.
When $p\not=q$, $G$ being \qe\ at $q$ implies
that the Sylow $p$-subgroup of $G$ is cyclic, and we may use the formula
from Theorem~\ref{BenardFeit} to compute $m_p(\chi)$.
We are reduced to the situation $p=q$.
When $p$ is odd, Theorem~\ref{localU1} implies that $m_p(\chi)=1$.

The last remaining problem is the case $p=q=2$. In \cite{Benard76} this case
is described as being ``much more complicated". During the 1970's various
special cases were dealt with (\cite{Lorenz, Yamada74, Yamada78} contain
examples).
We know by Theorem~\ref{localU1}
that $m_2(\chi)$ is either $1$ or $2$, with $1$ being the value when
$\fru\in\Q_2(\chi)$.
If $4\mid m(\chi)$, then $\fru\in\Q(\chi)$ by a result of Benard and
Schacher, and so $m_2(\chi) = 1$.
It follows that we may assume that the highest power of $2$ dividing $m(\chi)$
is either $1$ or $2$ and the full force of Theorem~\ref{counting} can be
applied, but this is not always enough to compute $m_2(\chi)$
(see Example \ref{example2}).
The article of Yamada \cite{Yamada82} gave a full solution to finding
$m_2(\chi)$. For our calculations we have chosen to use another solution.
We will use the methods of Schmid \cite{Schmid},
as corrected in Riese \& Schmid \cite{RieseSchmid}.
The rest of this section discusses these 2-adic methods.

After using the Lemma~\ref{subgpq} and reducing to a faithful character,
we may assume that we have the following situation.

\medskip\noindent
\club\ We have that $\chi\in\Irr(G)$ is faithful, and
for every proper subgroup $H<G$ and every $\eta\in\Irr(H)$,
$\ip<\eta^G,\chi>[\Q_2(\chi,\eta):\Q_2(\chi)]$ is divisible by 2.

\medskip\noindent
It turns out that condition \club\ implies that $m_2(\chi)$ only depends on
the isomorphism type of $G$. The cited works of Schmid and Riese \& Schmid
describe in detail those of these groups having $m_2(\chi) = 2$.

Let $Q_8$ denote the quaternion group of order
$8$, $Q_8 = \langle a,b\mid a^2=b^2, b^4, ba=ab^3\rangle$.
For an odd prime $r$, we say that a group $G$ is of type $(Q_8,r)$ when
the following all hold. 
The group $G$ is a split extension $G = U\split P$,
where $U$ has order $r$ and $P$ is a 2-group acting on $U$.
Putting $X=C_P(U)$, we require that $X\cong Q_8$, that $P$ is a central
product $P=XC_P(X)$, and that $P/X$ is non-trivial and cyclic.
We also require that the 2-part of the order of 2 modulo $r$ is equal
to $|P/X|$.
This definition was taken from \cite{Schmid} (3.5).

We need another class of groups. For odd prime $r$, we say that a group
$G$ is of type $(QD,r)$ when the following hold. The group $G$ is a split
extension $G= U\split P$, where $U$ has order $r$ and $P$ is a 2-group acting
on $U$.
Let $Z=P'$, the derived group of $P$. We require that $Z$ is cyclic of order
at least 4 and that $Y/Z$ is cyclic, where $Y$ is the centralizer in $P$ of an
element of order 4 in $Z$. Putting $X=C_P(U)$, we require that $X$ is not
abelian and $|P/X| = 2^s$, where $|P/Z| = 2^{s+1}$
(and $Z$ has index 2 in $X$). We further require that the order of
2 modulo $r$ has the form $2^sf$, where $f$ is an integer. We exclude from
type $(QD,r)$ groups where $f$ is odd and $X\cong Q_8$.
This definition is taken from \cite{RieseSchmid}, see \S4 Lemma~2 and \S5,
definition on p195.
The ``$QD$" in the name refers to the isomorphism type of $X$.
In these groups $X$ is either generalised quaternion or dihedral.

The following theorem is essentially a statement of results in
\cite{Schmid} and \cite{RieseSchmid}.
\begin{theorem}\label{twotwo}
Suppose that $G$ and $\chi$ satisfy \club. Then $m_2(\chi)=2$ if
and only if either $G\cong Q_8$ or there exists an odd prime $r$ such that
$G$ is of type $(Q_8,r)$ or of type $(QD,r)$.
\end{theorem}
\begin{proof}
Suppose that $m_2(\chi)=2$. In the terminology of \cite{RieseSchmid}, $G$ is
a dyadic Schur group. That $G$ has the form stated in the theorem is 
\cite{RieseSchmid} Lemma~4.  

Conversely, we claim that for every group $G$ of the given types,
and for every faithful $\chi\in\Irr(G)$, we have $m_2(\chi) = 2$.

If $G\cong Q_8$ then there is only one possibility for $\chi$ and it is well
known that this character has $m_2(\chi)=2$.
When $G$ is of type $(Q_8,r)$ it is shown in \cite{Schmid}
Theorem~4.3~(5) that every faithful irreducible character of $G$ has 
$m_2(\chi)=2$. 

Now suppose that $G$ is of type $(QD,r)$. As in the proof of
\cite{RieseSchmid} Theorem~5, and with
notation as in the above definition of this type, let $H = UZ$. Then $H$
is a cyclic normal subgroup of $G$ with $H = C_G(H)$. It is shown in the proof
of \cite{RieseSchmid} Theorem~5 that if $\lambda$ is a faithful linear
character of $H$ then $m_2(\lambda^G) = 2$. It is straightforward to check
that every faithful $\chi\in\Irr(G)$ is of the form $\chi = \lambda^G$,
where $\lambda$ is a faithful linear character of $H$, and we are done.
%
%Let $\lambda\in\Irr(H)$ be any linear constituent of $\chi_H$.
%Let $K$ be the kernel of $\lambda$. As $H\triangleleft G$ is cyclic,
%it follows that $K\triangleleft G$, and so $K$ is the kernel of $\lambda^G$.
%But $\chi$ is a constituent of $\lambda^G$,
%so $K$ is contained in the kernel of $\chi$, which is trivial, and we see that
%$\lambda$ is faithful. Since $\lambda$ is a faithful representation of $H$,
%the inertia group of $\lambda$ in $G$ is $C_G(H) = H$, hence $\lambda^G$ is
%irreducible and so $\lambda^G = \chi$.
\end{proof}

We look more closely at the reduction to \club. 
This reduction goes through some sequence of pairs $(G_i,\chi_i)$, for
$i = 0\dots n$, where $(G,\chi) = (G_0, \chi_0)$ and $(G_n,\chi_n)$
satisfies \club. We will use Theorem~\ref{twotwo} to determine $m_2(\chi_n)$.
Once we have done this we want to determine $m_2(\chi_i)$ from
$m_2(\chi_{i+1})$ for each $i < n$.
For each $i < n$ we have one of two possibilities.

The first possibility is the quotient case, where $G_{i+1} = G_i/\ker\chi_i$
and $\chi_{i+1}= \hat\chi_i$ is the character of $G_{i+1}$ that lifts to
$\chi_i$.
In this case we have $m_2(\chi_i) = m_2(\chi_{i+1})$.

The other possibility is that $G_{i+1}$ is a subgroup of $G_i$,
$\chi_{i+1}\in\Irr(G_{i+1})$, and 
$$\ip<\chi_{i+1}^{G_i},\chi_i>[\Q_2(\chi_i,\chi_{i+1}):\Q_2(\chi_i)]$$
is odd.  In this case we have
$$m_2(\chi_i) = m_{\Q_2(\chi_i)}(\chi_{i+1}) =
\frac{m_2(\chi_{i+1})}
{\gcd(m_2(\chi_{i+1}), [\Q_2(\chi_i,\chi_{i+1}):\Q_2(\chi_{i+1})])}.$$
So $m_2(\chi_{i+1})=1$ gives $m_2(\chi_i)=1$, and
$d_i=[\Q_2(\chi_i,\chi_{i+1}):\Q_2(\chi_{i+1})]$ even will also give
$m_2(\chi_i)=1$. On the other hand, $d_i$ odd gives
$m_2(\chi_i) = m_2(\chi_{i+1})$.

While computing the reduction to \club\ we determine $d_i$ at each subgroup
step. If any $d_i$ is even, then the original $m_2(\chi)$ will be $1$,
so we can end the computation.
Otherwise the original $m_2(\chi)$ will be equal to the final $m_2(\chi_n)$.

%%%%%%%%%%%%%%%%%%%%%%%%%%%%%%%%%%%%%%%
\section{Algorithm sketch and examples} \label{sec:algsketch}
\begin{enumerate}
\item Determine $\m_i(\chi)$ and use Theorem~\ref{globalU} to find a bound
$u\in\N$ such that $m(\chi)$ divides $u$. If $u=1$ terminate.
\item For each rational prime $p$ dividing the order of $G$ use
Theorems~\ref{localU1}, \ref{localU2}, \ref{BenardFeit} and \ref{counting}
to find a $u_p$ dividing $u$ such that $m_p(\chi)$ divides $u_p$.
Set $u$ to be the LCM of $\m_i(\chi)$ and the $u_p$'s. If $u=1$ terminate.
\item For each prime $p$ with $u_p > 1$, and each $q$ with $q$ dividing $u_p$ do
the following steps.
\begin{enumerate}
\item Locate a \qe\ subgroup
$H\le G$ and $\eta\in\Irr(H)$ such that $q$ does not divide
$\ip<\eta^G,\chi>[\Q_p(\chi,\eta):\Q_p(\chi)]$.
\item Determine $m_p(\eta)$.
\item Divide $u_p$ by a suitable power of $q$ so that $q$-part of $u_p$ becomes
equal to
$$m_{\Q_p(\chi)}(\eta) =
\frac{m_p(\eta)}{(m_p(\eta),[\Q_p(\chi, \eta):\Q_p(\eta)])}.$$
\end{enumerate}
\item At this point each $u_p = m_p(\chi)$. Return $\m_i(\chi)$ and
the sequence of pairs
$\langle p, u_p\rangle$ where $u_p\ne1$.
\end{enumerate}

\subsection{An example}\label{example1}
We consider the degree 48 irreducible character $\chi$ of the Coxeter group
$G=H_4$. This character is rational with indicator 1.
After the first step above we have $\m_i(\chi)=1$ and $u = 2$. 
The order of the group is $2^8 3^2 5^2$, so we must determine local Schur
indices for $p=2,3,5$. We make progress in step 2 by noticing that the 3-block
containing $\chi$ has defect 1, hence has cyclic defect groups,
and Benard's formula gives $m_3(\chi) = 2$.
If all we want is the value of $m(\chi)$, we could stop
now with the answer 2. Otherwise we have $u=u_2=u_5=2$ with $\m_i(\chi)=1$ and
$m_3(\chi)=2$. Feit's character table upper bounds for $m_2(\chi)$ and 
$m_5(\chi)$ give us nothing more.

We do step 3 with $p = 5$ and $q = 2$. Our first attempt
is to search through cyclic subgroups of $G$ and see if they have the potential
to be a cyclic normal subgroup of some useful \qe\ group, as in \S3.
We find a self-centralizing cyclic group $A$ of order 12 with a faithful linear
character $\eta$ such that $\ip<\eta^G, \chi>=7$. We have 
$\Q_5(\eta) = \Q_5(12)$. From the power map of $G$ we see that
there is an element of $G$ that conjugates each element of $A$ to its 5th power.
If we compare this with the Galois group of $\Q_5(12)/\Q_5$,
which has order 2 and interchanges $\zeta_{12}$ and $\zeta_{12}^5$,
it follows that there is a \qe\ at 2 subgroup $U$ of $G$, with order
$24$, containing $A$ as normal subgroup, with $C_U(A) = A$, and
such that $\eta^U$ satisfies the conditions of Lemma~\ref{subgpq} with $q=2$
and ground field $F = \Q_5$.
(In detail, $C_U(A) = A$ implies that $\eta^U$ is irreducible,
$\Q_5(\eta^U) = \Q_5$, and
$\ip<(\eta^U)^G,\chi> = 7$.)
Since 5 does not divide the order of $U$
we have $m_5(\eta^U) = 1$, so $m_5(\chi) = 1$. Now Theorem~\ref{counting}
leaves no choice but $m_2(\chi)=2$, and we are done.

In this example Theorem~\ref{counting} determined $m_2(\chi)$.
If we did not use this result we would find that $G$ contains a subgroup $V$
of order 48, having $U$ as a subgroup of index 2, such that $\eta^V$ is
irreducible and rational, so satisfying Lemma~\ref{subgpq} with $q=2$ and
$F = \Q$ (and hence for all ground fields of interest to us).
In addition $\eta^V$ is faithful and $V$ is of type $(Q_8,3)$,
so $m_2(\eta^V) = 2$, giving $m_2(\chi) = 2$.

\subsection{A \qe\ example}\label{example2}
Let $G$ be the following group of order 272.
\begin{eqnarray*}
G=\langle\, a,b,c,d,e\mid a^2, b^2, c^2 = d, d^2, e^{17}, b^a = bc, c^a=cd,
c^b =cd, e^a=e^{-1}, \\
  (a,d), (b,d), (b,e), (c,e) \,\rangle.
\end{eqnarray*}
(Using \Magma, $G$ is obtained as \texttt{SmallGroup(272,15)}.)
Let $\chi$ be any faithful element of $\Irr(G)$. Then $\chi$ has degree 4,
indicator 1, and $\Q(\chi) = \Q(\zeta_{17}+\bar\zeta_{17})$.
The prime divisors of the order of $G$ are 2 and 17.
After step 2 we have upper bounds $u_2=u_{17}=2$ with all others $1$.
This $G$ is \qe\ at $2$ and the Sylow 17-subgroup of $G$ is cyclic.
We find that $m_{17}(\chi) = 1$ by using Benard's formula.
Theorem~\ref{counting} gives us no help in determining
$m_2(\chi)$ as $\Q(\chi)$ has two places over $2$. We find that condition
\club\ is satified by this $\chi$, and that $G$ is of type $(QD,17)$,
concluding that $m_2(\chi)=2$.

\subsection{ATLAS and non-ATLAS}\label{error}
We take $G$ to be a group of form $6.M_{22}.2$, a bicyclic extension of
$M_{22}$, where the outer $2$ inverts the generator of the cyclic $6$.
We first take the group featured in the ATLAS, where $G$ is a split extension
of $6.M_{22}$.
As $\chi$ we consider a faithful character of degree $420$ with $\Q(\chi) =\Q$.
This character has indicator $1$, so $\m_i(\chi)=1$ and $u=2$.
The prime divisors of the group order are $2,3,5,7,11$.
The Sylow $p$-subgroups for $p=5,7,11$ are cyclic
and we calculate that $m_p(\chi)=1$ for these $p$ using Benard's formula.
There is an $x\in G$ with order $33$ and $\chi(x) = -1$.
If $B$ is the $2$-block containing $\chi$, we have
$\Q(\psi(x):\psi\in\Irr(B)) = \Q(\sqrt{33}) \subseteq \Q_2$.
We apply Feit's bound (Theorem~\ref{localU2}) with $p=2$ to get
$m_2(\chi) = 1$. Now Theorem~\ref{counting} gives $m_3(\chi)=1$,
so $m(\chi) = 1$.

Now take $G$ to be the group of the given form that is isoclinic and not
isomorphic to the ATLAS group above.
In this group there are no outer elements of order 2, and an outer element of
order 4 squares to the element of order $2$ in the cyclic 6.
Take $\chi$ to be the character corresponding to the character above.
As $\chi$ is zero on outer elements we still have $\Q(\chi) = \Q$.
However, the indicator is now $-1$, so $\m_i(\chi) = 2$.
The same arguments show that $m_p(\chi) =1$ for
$p\ge5$ and $p=2$. This time Theorem~\ref{counting} gives $m_3(\chi)=2$.

%%%%%%%%%%%%%%%%%%%%%%%%%%%%%%
\section{Some implementation details}\label{sec:algdetails}
In this section we give some details of how various parts of the algorithm
have been implemented.
\subsection{Cyclotomic extensions of $\Q_p$}\label{cyclogal}
%\begin{picture}(6,6)
%\put(3,1){\circle*{0.1}}
%\put(3,1){\line(1,1){2}}
%\qbezier[20](3,1)(4,2)(5,3)
%\put(3,1){\line(-1,1){2}}
%\put(3,5){\circle*{0.1}}
%\put(3,5){\line(-1,-1){2}}
%\qbezier[20](3,5)(2,4)(1,3)
%\put(3,5){\line(1,-1){2}}
%\put(1,3){\circle*{0.1}}
%\put(5,3){\circle*{0.1}}
%\put(3,5){\makebox(0,0)[b]{$\strut\Q_p(N)$}}
%\put(3,1){\makebox(0,0)[t]{$\strut\Q_p$}}
%\put(1,3){\makebox(0,0)[r]{$\Q_p(p^k)\;$}}
%\put(5,3){\makebox(0,0)[l]{$\;\Q_p(m)$}}
%\end{picture}
During the course of a run through the algorithm, we must deal with fields of
the form $\Q_p$ extended by elements of $\Q(N)$ (particularly character values).
This need arises when using the Brauer-Witt Theorem, Benard's formula or
reducing to \club\ for instance.
We do this using the Galois groups of $\Q_p(N)/\Q_p$, their
subgroups, and the Galois correspondence. Note that we don't need a 
general algorithm for membership testing in subfields of $\Q_p(N)$, just for
elements of $\Q(N)$.

Let $N = p^km$, where $m$ is the $p'$-part of $N$.
The Galois group $\G(\Q_p(N)/\Q_p)$ is the direct product of
\begin{itemize}
\item the group of units modulo $p^k$, where $u$ takes $\zeta_{p^k}\mapsto
\zeta_{p^k}^u$, fixing $\zeta_m$, and
\item a cyclic group fixing $\zeta_{p^k}$ and with a generator taking
$\zeta_m\mapsto\zeta_m^p$.
\end{itemize}
To compute a subgroup corresponding to a field extension by an
element of $\Q(N)$, we use the orbit-stabilizer method.
The degree of an extension is the subgroup index, and deciding field
membership is done by testing if the element of $\Q(N)$
is fixed by the corresponding subgroup.  The implementation uses \Magma's
abelian group and cyclotomic field types for these calculations.

\subsection{Search for useful \qe\ subgroups}\label{subgroups}
This is a very important step in the algorithm.
When a search is needed it can take
most of the time for the computation. The search is currently in two phases.
The first is as in Example \ref{example1}, where we search the cyclic subgroups
hoping to use the power map to find one that can be extended to give a useful
\qe\ subgroup/character pair $H,\eta$ for the calculation.
This search requires $\eta$ to be induced from the cyclic subgroup.
In Example \ref{example1} we didn't have to construct
the subgroup to determine the Schur index we needed. Even when we find
suitable cyclic group and character we are not always so lucky.

There are characters where the search through cyclic groups finds nothing
useful.
If this happens we construct a set of subgroups of $G$, $\E_q$, consisting of
groups \qe\ at $q$, and including all maximal such subgroups of
$G$ up to conjugacy.
Let $\{C_i:i\in I\}$ be a set of conjugacy class representatives for all cyclic
subgroups of $G$ having $q'$ order. For each $i$ let $S_i$ be some
Sylow $q$-subgroup of $N_G(C_i)$. It is clear that $C_iS_i$ is
a subgroup of $G$ which is \qe\ at $q$.
Now put $\E_q = \{C_iS_i:i\in I\}$. It is straightforward to see that
some conjugate of every \qe\ at $q$ subgroup of $G$ is contained
in some element of $\E_q$.

If the cyclic group search gives us nothing useful then we search through
the groups in $\E_q$.
If none of these have an irreducible character $\eta$ as in Theorem~\ref{BW}
with $F = \Q$, we conclude that the $q$-part of $m(\chi)$ is 1.
Otherwise we have reduced to a group which is \qe\ at $q$.
In all cases except $p=q=2$ we can now compute the the $q$-part of
$m_p(\chi)$ as sketched at the beginning of \S\ref{sec:compute_m_p}.

If we are in the situation $p=q=2$, we have further reductions to make.
We reduce to a section of $G$ so that $\eta$ is faithful.
For these calculations we use a polycyclic presentation of the
group which allows easy representation of quotients as well as
subgroups. We further search for a maximal subgroup of $G$ and an
irreducible character of the subgroup, testing \club. 

\subsection{Finding $m_p(\chi)$}
The methods used here are covered in \S\ref{sec:compute_m_p}.
We organise the loop over $p$ so that $p=2$ comes last.
We only use Theorem~\ref{twotwo}
when other methods fail to give us $m_2(\chi)$. We make some points about
the computations needed when using Theorem~\ref{twotwo}.
First we must ensure that \club\ is satisfied.
Although this condition is stated in terms of
all proper subgroups of $G$, Lemma~\ref{supergroup} allows us
to check the condition for $H$ running over the maximal
subgroups of $G$ only.
If we find a maximal subgroup $H$ and $\psi\in\Irr(H)$ with
$\ip<\psi^G,\chi>[\Q_2(\chi,\psi):\Q_2(\chi)]$ odd, we consider
$d=[\Q_2(\chi,\psi):\Q_2(\psi)]$. If $d$ is even we have $m_2(\chi)=1$.
Otherwise we replace $G,\chi$ with $H,\psi$ and repeat with the maximal
subgroups of the new $G$.

Once we reach a state where \club\ holds (and so Theorem~\ref{twotwo} applies),
determining whether or not $G$ is $Q_8$ or is of type $(Q_8,r)$ or type
$(QD,r)$ is a simple exercise in computing with polycyclic groups,
testing the given definitions of these types.

Yamada's solution \cite{Yamada82} for the $p=q=2$ situation has been
programmed and tested alongside the solution given above.
It was reassuring that in all tests the two
methods have given the same answers, but the implementation of the Yamada
method has proved a little slower than the method given above.

\section{Tests}\label{sec:tests}
The program was tested against the tables of Schur indices of characters of
groups related to sporadic simple groups in Feit's
article \cite{Feit96}. The groups tested were all the decorated versions of 
24 of the sporadic simple groups, excluding $B$ and $M$. 

The groups used for computations were permutation and matrix groups
taken from \Magma's collection supplied by R.A. Wilson from \cite{Wilsonetal}, 
with \textit{ad hoc} adaptations to get groups isoclinic to the ATLAS groups.

The test was a calculation of the local Schur indices of all
characters of the groups.
Non-faithful characters were included, so in many cases these
gave two calculations with essentially the same characters, but in different
groups.

There were a small number of differences between our results and those of
\cite{Feit96}, the list follows. Each item in the list starts with a line
giving the group, character in Feit's notation, character degree,
character field, and Frobenius-Schur indicator.
Each character is faithful on the group named.
Following this line is the result of our calculation with a justification.

\begin{itemize}
\item $6.M_{22}.2i$, $48^*$, 420, $\Q$, $-1$. \newline
Feit states that $m(\chi) = 1$.
We found that $\m_i(\chi)=m_3(\chi)=2$ with all other local indices being $1$.
Our calculations are as in \S\ref{error}, the non-ATLAS case.

\item $3.McL.2$, $34^*$, $9\,504$, $\Q$, $1$. \newline
Feit states that $m(\chi) =1$.
The program finds $m_2(\chi)=m_3(\chi)=2$, with all others equal to 1.
In this case the cyclic subgroups search finds a \qe\ subgroup $H$,
with $|H|= 48$, and a faithful, rational character $\eta\in\Irr(H)$ having
$\ip<\eta,\chi_H>=805$ so that $m_3(\eta) = m_3(\chi)$.
Benard's formula gives $m_3(\eta) = 2$, Theorem~\ref{counting} 
gives $m_2(\chi)=2$.

\item $2.Suz.2$, $54^*=55^*$, $70\,200$, $\Q$, $1$.\newline
The program finds that $m(\chi) = 1$.
The tables on p249 of \cite{Feit96} give $m_3(\chi)=m_7(\chi)=2$, but this is
only an error in the table, as the argument Feit gives on p241 shows that
$m(\chi)=1$.

\item $2.Suz.2$, $56^* = 57^*$, $70\,200$, $\Q$, $-1$. \newline
Feit gives $\m_i(\chi)=m_3(\chi)=2$ while we find $\m_i(\chi)=m_7(\chi)=2$.
We have $\m_i(\chi) = 2$. If $x\in G$ has order $21$, then $\chi(x) = -1$ and
the class of $x$ is integral, so $m_p(\chi)=1$ for all $p\not=3,7,\infty$
(Theorem~\ref{localU2}) and $1\le m_3(\chi),m_7(\chi)\le2$.
The Sylow 7-subgroup of $G$ is cyclic and Benard's formula gives $m_7(\chi)=2$.
It follows from Theorem~\ref{counting} that $m_3(\chi)=1$.

\item $2.Suz.2i$, $62^*=63^*$, 159\,744, $\Q$, $-1$.\newline
Feit gives $\m_i(\chi)=m_3(\chi)=2$, while our program finds
$\m_i(\chi)=m_2(\chi)=2$, with $m_3(\chi) = 1$.
We confirm this result by taking $C$ to be
a cyclic group of order 24 in $2.Suz$, such that a generator maps to
class 12D in $Suz$. Let $S$ be the Sylow $2$-subgroup of the normalizer
of $C$ in $2.Suz.2i$, and let $H=CS$. Then $|H| = 2^6\cdot3$ and there is
a rational $\eta\in\Irr(H)$ with degree $4$ and $\ip<\eta,\chi_H>=6657$.
The Sylow 3-subgroup of $H$ is contained in $\ker\eta$,
so $\eta$ is a character of a 2-group. We deduce $m_3(\eta)=1$ and
then apply Lemma~\ref{twogps}. As $\eta$ is rational with indicator 
$-1$, we get $m_2(\eta)=\m_i(\eta)=2$.
It follows that $m_3(\chi)=1$ and $m_2(\chi)=2$.

\item $Ly$, $37$, 36\,887\,520, $\Q$, $1$. \newline
In \cite{Feit83} it is stated that all characters of $G$ have Schur index 1.
We find that $m_2(\chi) = m_7(\chi) = 2$, with all other $m_p(\chi)$
equal to 1. We have $\m_i(\chi)=1$.  Further, $m_p(\chi)=1$ for all
$p\not\in\{2,3,5,7,11\}$ as these are the prime divisors of $|G|/\chi(1)$.
The Sylow 7-subgroup of $G$ is cyclic, and Benard's formula gives
$m_7(\chi) = 2$. 
The values of $\chi$ on elements of order 11 and 21 ($\chi(x)=\pm 1$)
show that $m_3(\chi) = m_5(\chi) = m_{11}(\chi) = 1$.
Note that these elements do not make up integral classes, so the side
conditions of Feit's upper bound are non-trivial here.
We then must have $m_2(\chi) = 2$.

\item $3.ON.2$, $38^*$ and $39^*$, 105\,336, $\Q(\sqrt{7})$, $1$. \newline
Feit claims that these characters have Schur index 1, while
we find $m_3(\chi) = 2$.
Using the character table of $G$ and Theorem~\ref{localU2}, we find that
for all $p\not= 3$ we have $m_p(\chi) = 1$.
Note that the character field has two places over 3, and is a subfield
of $\Q_3$. Let $H$ be a maximal subgroup of $G$ with order 967\,680.
Let $\eta$ be the unique member of $\Irr(H)$ with degree 144.
The inner product of $\chi$ with $\eta^G$ is 21, and $\eta$ is rational,
so $m_3(\chi) = m_3(\eta)$.
The 3-block of $\eta$ has defect 1, and we compute $m_3(\eta) = 2$ using
Benard's formula.

\item $J_4$, $50$, 1\,579\,061\,136, $\Q$, $1$. \newline
Feit claims $m(\chi) = 1$, while we find $m_3(\chi) = m_5(\chi) = 2$.
Using the character table of $G$ and Theorem~\ref{localU2}, we find that,
for all $p\not= 3,5$ we have $m_p(\chi) = 1$. The Sylow 5-subgroup of
$G$ is cyclic, and Benard's formula gives $m_5(\chi) = 2$.
That $m_3(\chi) = 2$ now follows.
\end{itemize}

For the groups $B$, $2.B$ and $M$ the character tables in the GAP character
table library \cite{GAPChtrs} were used to check Feit's results.
No discrepancies were found.
It was fortunate that no subgroups were needed to decide all the Schur indices.

\bibliographystyle{plain}
\bibliography{schur_index}

\begin{thebibliography}{10}

\bibitem{Benard72}
M.~Benard.
\newblock The {S}chur subgroup {I}.
\newblock {\em J. Algebra}, 22:374--377, 1972.

\bibitem{Benard76}
M.~Benard.
\newblock Schur indices and cyclic defect groups.
\newblock {\em Ann. of Math.}, 103:283--304, 1976.
\newblock Reviewed by W. Feit, MR0412265.

\bibitem{MagmaLanguage}
W.~Bosma, J.~Cannon, and C.~Playoust.
\newblock The {M}agma {A}lgebra {S}ystem {I}: {T}he {U}ser {L}anguage.
\newblock {\em J. Symbolic Comput.}, 24:235--265, 1997.

\bibitem{Brauer51}
R.~Brauer.
\newblock On the algebraic structure of group rings.
\newblock {\em J. Math. Soc. Japan}, 3:237--251, 1951.

\bibitem{GAPChtrs}
T.~Breuer.
\newblock The {GAP} {C}haracter {T}able {L}ibrary, 2017.
\newblock
  \newline\verb=http://www.math.rwth-aachen.de/~Thomas.Breuer/ctbllib/=.

\bibitem{ATLAS}
J.~H. Conway, R.~T. Curtis, S.~P. Norton, R.~A. Parker, and R.~A. Wilson.
\newblock {\em {ATLAS} of finite groups}.
\newblock Oxford University Press, Eynsham, 1985.

\bibitem{Feit82}
W.~Feit.
\newblock {\em The representation theory of finite groups}.
\newblock North-Holland, 1982.

\bibitem{Feit83}
W.~Feit.
\newblock The computation of some {S}chur indices.
\newblock {\em Israel J. of Math.}, 46:274--300, 1983.

\bibitem{Feit96}
W.~Feit.
\newblock Schur indices of characters of groups related to finite sporadic
  simple groups.
\newblock {\em Israel J. of Math.}, 93:229--251, 1996.

\bibitem{Huppert}
B.~Huppert.
\newblock {\em Character {T}heory of {F}inite {G}roups}.
\newblock de Gruyter, 1998.

\bibitem{Isaacs}
I.M. Isaacs.
\newblock {\em Character {T}heory of {F}inite {G}roups}.
\newblock Dover, 1994.

\bibitem{LangANT}
S.~Lang.
\newblock {\em Algebraic number theory}, volume 110 of {\em Graduate Texts in
  Mathematics}.
\newblock Springer-Verlag, New York, second edition, 1994.

\bibitem{Lorenz}
F.~Lorenz.
\newblock \"{U}ber die berechnung der {S}churschen indizes von charakteren
  endlicher gruppen.
\newblock {\em J. Number Theory}, 3:60--103, 1971.

\bibitem{Lorenz72}
F.~Lorenz.
\newblock Charaktere endlicher {G}ruppen mit vorgegebenen {S}churschen
  {I}ndizes.
\newblock {\em Math. Ann.}, 195:315--320, 1972.

\bibitem{Navarro}
G.~Navarro.
\newblock {\em Characters and blocks of finite groups}, volume 250 of {\em
  London Mathematical Society Lecture Note Series}.
\newblock Cambridge University Press, Cambridge, 1998.

\bibitem{Pierce}
R.S. Pierce.
\newblock {\em Associative Algebras}.
\newblock Springer-Verlag, New York, 1982.
\newblock Graduate Texts in Mathematics 88.

\bibitem{Plesken}
W.~Plesken.
\newblock {\em Group rings of finite groups over $p$-adic integers}.
\newblock Springer-Verlag, Berlin, 1983.
\newblock Lecture Notes in Mathematics, Vol. 1026.

\bibitem{RieseSchmid}
U.~Riese and P.~Schmid.
\newblock Schur indices and {S}chur groups, {II}.
\newblock {\em J. Algebra}, 182:183--200, 1996.

\bibitem{Schmid85}
P.~Schmid.
\newblock Representation groups for the {S}chur index.
\newblock {\em J. Algebra}, 97:101--115, 1985.

\bibitem{Schmid}
P.~Schmid.
\newblock Schur indices and {S}chur groups.
\newblock {\em J. Algebra}, 169:226--247, 1994.

\bibitem{Turull}
A.~Turull.
\newblock Schur indices of perfect groups.
\newblock {\em Proc. Amer. Math. Soc.}, 130:367--370, 2002.

\bibitem{UngerChtr}
W.~R. Unger.
\newblock Computing the character table of a finite group.
\newblock {\em J. Symbolic Comput.}, 41:847--862, 2006.

\bibitem{Wilsonetal}
R.A. {Wilson \emph{et al}}.
\newblock {ATLAS} of finite group representations, 2007.
\newblock \newline\verb=http://brauer.maths.qmul.ac.uk/Atlas/=.

\bibitem{Witt52}
E.~Witt.
\newblock Die algebraische struktur des gruppenringes einer endlichen gruppe
  \"uber einem zahlk\"orper.
\newblock {\em J. Reine Angew. Math.}, 190:231--245, 1952.

\bibitem{Yamada74}
T.~Yamada.
\newblock {\em The {S}chur subgroup of the {B}rauer group}.
\newblock Springer-Verlag, Berlin, 1974.
\newblock Lecture Notes in Mathematics, Vol. 397.

\bibitem{Yamada78}
T.~Yamada.
\newblock Schur indices over the $2$-adic field.
\newblock {\em Pacific J. Math.}, 74:273--275, 1978.

\bibitem{Yamada79}
T.~Yamada.
\newblock A remark on {S}chur indices of $p$-groups.
\newblock {\em Proc. Amer. Math. Soc.}, 76:45, 1979.

\bibitem{Yamada82}
T.~Yamada.
\newblock The {S}chur index over the 2-adic field.
\newblock {\em J. Math. Soc. Japan}, 34:307--315, 1982.

\end{thebibliography}

\noindent
Address:\\[5mm]
School of Mathematics and Statistics\\
University of Sydney\\
NSW 2006\\
Australia\\
e-mail: william.unger@sydney.edu.au\\[5mm]

\end{document}